\documentclass[10pt]{article}
\usepackage{mathrsfs}
\usepackage{amsfonts}
\usepackage{titlesec}
\usepackage{enumerate}
\usepackage[toc,page,title,titletoc,header]{appendix}
\usepackage{amsthm,amsmath,amssymb,anysize}
\usepackage[numbers,sort&compress]{natbib}
\theoremstyle{definition}
\newtheorem{theorem}{Theorem}
\newtheorem{lemma}{Lemma}[section]

\newtheorem{proposition}[lemma]{Proposition}
\newtheorem{definition}[lemma]{Definition}
\newtheorem{corollary}[lemma]{Corollary}
\newtheorem{example}[lemma]{Example}
\newcommand{\periodafter}[1]{#1.}
\titleformat{\subsection}[runin]
{\normalfont\bfseries}{\thesubsection}{0.5em}{\periodafter}
\marginsize{4.2cm}{4.2cm}{3.6cm}{3cm}
\numberwithin{equation}{section}

\renewcommand{\thesubsection}{\arabic{section}.\arabic{subsection}}

\newcounter{RomanNumber}

\DeclareMathOperator\rank{\mathrm{rank}}

\title{Non-restricted representations of  contact and special contact Lie superalgebras of odd type\footnote{Supported by the NSF of China (12061029) and the NSF of Heilongjiang
Province (YQ2020A005).}}
\author{\textsc{ Shujuan Wang$^1$,
  Wende Liu$^{2}$\footnote{Correspondence:  wendeliu@ustc.edu.cn (W. Liu).}}\\
{\small \textit{$^1$Department of Mathematics, Shanghai Maritime University,}}\\
\small \textit{Shanghai 201306, China}  \\
\small\textit{$^2$School of Mathematics and Statistics,}
\textit{Hainan Normal University,}\\ \small\textit{ Haikou 571158,  China} }
\date{ }

\begin{document}
\maketitle
\begin{quotation}
\small\noindent \textbf{Abstract}:
Let $\frak{g}$ be a contact Lie superalgebra of odd   type
or special contact Lie superalgebra of odd   type
over an algebraically closed field of  characteristic $p>3$.
In this paper we  study   non-restricted representations of
$\frak{g}$. By using induced Kac modules,
we characterize all simple $\frak{g}$-modules with nonsingular
or $\Delta$-invertible  $p$-characters.
We also obtain all simple $\frak{g}$-modules with regular semisimple
$p$-characters.

\vspace{0.2cm} \noindent{\textbf{Keywords}}: Induced Kac module;
non-restricted representation;  contact and special contact Lie superalgebras of odd type

\vspace{0.2cm} \noindent{\textbf{Mathematics Subject Classification 2010}}: 17B10,   17B35, 17B50
\end{quotation}

\setcounter{section}{-1}
\section{Introduction}
Analogous to finite-dimensional
 simple Cartan type modular Lie algebras,
four families of  Cartan type modular Lie superalgebras
(that is, the Witt type, the special type, the Hamiltonian type and
 the contact type)
 were constructed  by Zhang in \cite{Z}.
 Later another four families of finite-dimensional
simple Cartan type modular Lie superalgebras
(that is,
the Hamiltonian Lie superalgebras of odd  type,
the  contact Lie superalgebras  of  odd type,
the special Hamiltonian Lie superalgebras of odd  type
and the special contact Lie superalgebras of odd type) are introduced by Fu, Liu and Zhang etc. in \cite{LZ,FJZ,LH,BL},
which  occur neither in the non-super case, nor in the non-modular case
(among finite-dimensional algebras).
Quite recently, Shu and Zhang determined all
restricted and some non-restricted simple modules of the Witt type modular Lie superalgebras
by virtue of Serganova's  odd reflections  together with Jantzen's $u(\frak{g})$-$T$ module
category (see \cite{SZ,SZ1}).
Later, Shu, Yao and Duan
studied  modular representations of  the first four kinds of  Cartan type modular  Lie superalgebras (see \cite{Y1,Y2,YS,SY,Duan,Duan1}).
In particular, they characterized some non-restricted representations of Witt type, special
type and Hamiltonian type modular Lie superalgebras in \cite{Y2,Duan,Duan1}.
This paper is a sequel to \cite{YL,WLKO}, in which
restricted representations of the Hamiltonian and contact Lie superalgebras of odd type are studied.
As far as I know, non-restricted representations of
the last four series of Cartan type modular Lie superalgebras have not been studied yet until now.
Let $\frak{g}$ be a contact or special contact Lie superalgebra of odd   type,
whose superderivations and maximal graded subalgebras have been studied in \cite{BL} and \cite{LWmax} respectively.
In this paper,  we characterize all simple $\frak{g}$-modules of nonsingular
or $\Delta$-invertible $p$-characters
and determined all simple $\frak{g}$-modules of regular semisimple
$p$-characters. In particular,
all those simple modules mentioned above are proved to
be reduced Kac modules.

\section{Preliminaries}
\subsection{Divided power superalgebras}
Let $\mathbb{F}$ be an algebraically closed field of characteristic  $p>3$
and
$\mathbb{Z}_2= \{\bar{0},\bar{1}\}$  the additive group of two
elements.
Fix  a pair of positive integers $m, n$  and write $\underline{r}=(r_1,\ldots,r_n\mid r_{n+1},\ldots,r_{n+m})$ for a $(m+n)$-tuple of non-negative integers.
Let $\mathbf{I}(n,  m)$ be the set of all $(m+n)$-tuples $\underline{r}$ such that
$ 0\leq r_i<p$ for $ i\leq n $ and $ r_i=0$ or $1$ for $n\leq i\leq m+n.$
Write $\mathcal{O}(n,  m)$  for the  divided power superalgebra (see \cite[p. 226]{Leb}),
which is a  supercommutative associative algebra  having a $\mathbb{Z}_{2}$-homogeneous   basis:
 $
\left\{x^{(\underline{r})}\mid \underline{r}\in \mathbf{I}(n,  m)\right\}
$
with parity
 $ |x^{(\underline{r})}|=\left(\sum_{i>n}r_{i}\right)\bar{1}$
   and
 multiplication relation:
 $$
 x^{(\underline{r})} x^{(\underline{s})}=\prod_{i=m+1}^{m+n}\min(1, 2-r_i-s_i)(-1)^{\sum_{n<i<j\leq m+n}r_js_i}\left(
\begin{array}{c}
\underline{r}+\underline{s} \\
 \underline{r}
\end{array}\right)x^{(\underline{r}+\underline{s})},
 $$
where
$$
\left(\begin{array}{c}
\underline{r}+\underline{s} \\
 \underline{r}
\end{array}\right)=\prod_{i=1}^n\left(\begin{array}{c}
r_i+s_i \\
 r_i
\end{array}\right).
$$
The divided power superalgebra $\mathcal{O}(n, m)$  is a generalization of the divided power algebra $\mathcal{O}(n)$ and  is isomorphic to the tensor product of the divided power algebra with the trivial $\mathbb{Z}_{2}$-grading and the exterior algebra $\Lambda(m)$ with the natural $\mathbb{Z}_{2}$-grading: $\mathcal{O}(n)\otimes_{\mathbb{F}}\Lambda(m).$

\subsection{General vectorial Lie superalgebras}
For $i=1, \ldots, n+m$, letting $\varepsilon_i$ be the $(m+n)$-tuple with 1 at the $i$-th slot and 0  elsewhere,
 we write $x_i$ for $x^{(\varepsilon_i)}$ and
define the distinguished partial derivatives $\partial_i$  by letting
$$
\mbox{$\partial_i(x_j)=\delta_{ij}$, for $j=1, \ldots, n+m$}
$$
and extending $\partial_i$ to $\mathcal{O}(n, m)$ to   a superderivation of parity $|x_{i}|$.
The Lie superalgebra of all superderivations of $\mathcal{O}(n, m)$ contains an important subalgebra, called
 the general vectorial Lie superalgebra of distinguished superderivations (a.k.a. the Lie superalgebra of Witt type), denoted by $\mathfrak{vect}(n, m)$ (a.k.a.
$W(n,  m)$), having an $\mathbb{F}$-basis (see \cite[p. 226--227]{Leb})
$$
\left\{x^{(\underline{r})}\partial_k\mid \underline{r}\in \mathbf{I}(n,  m), 1\leq k\leq m+n\right\}.
$$
Generally speaking, $\mathfrak{vect}(n,  m)$ contains various finite-dimensional simple Lie superalgebras, which are analogous to  finite-dimensional simple modular Lie algebras or infinite-dimensional simple Lie superalgebras of vector fields over $\mathbb{C}$.

\subsection{Contact Lie superalgebras of odd type}
From now on set $m=n+1$.
As in \cite[p. 247--249]{Leb} and \cite[p. 35--42]{Lei},
for any $f\in\mathcal{O}(n, n+1), $
write
$$
\mathrm{M}_{f}=(2-E)(f)\partial_{2n+1}-\mathrm{Le}_{f}-(-1)^{|f|}\partial_{2n+1}(f)E,
$$
where
$E=\sum^{2n}_{i=1}x_{i}\partial_{i}$ is the Euler operator and
$\mathrm{Le}_{f}=\sum_{i=1}^{2n}(-1)^{|\partial_{i}||f|}\partial_{i}(f)\partial_{i^{'}}$
is the Hamiltonian vector field
with
\[ i'=\left\{
 \begin{array}{ll}
i+n,&\mbox{if  $i\leq n$ }\\
i-n,&\mbox{if $i>n.$}
\end {array}
\right.
\]
It is clear that $\mathrm{M}_{f}\in \mathfrak{vect}(n, n+1)$
for any $f\in\mathcal{O}(n, n+1).$
Note that $|\mathrm{M}_{f}|=|f|+\overline{1}$
and
\begin{equation*}\label{hee1.1}
   [\mathrm{M}_{f},\mathrm{M}_{g}]=\mathrm{M}_{\{f,g\}_{m.b.}}\;\mbox{for}\;
   f,g\in\mathcal{O}(n, n+1),
\end{equation*}
where $\{\cdot,\cdot\}_{m.b.}$ is the contact bracket given by
$$
{\{f,g\}_{m.b.}}=\mathrm{M}_f(g) + (-1)^{|f|}2 \partial_{2n+1}(f)(g).
$$
Then
$
\frak{m}:=\{\mathrm{M}_{f}\mid f\in\mathcal{O}(n,   n+1)\}
$
is a finite-dimensional simple  Lie subsuperalgebra of $\mathfrak{vect}(n,   n+1)$,
which is called   the
contact Lie superalgebra of odd type.
Note that $\frak{m}$
  is also called the  odd contact Lie superalgebra
  and  denoted by $KO(n,n+1, \underline{1})$ in  \cite[p. 111]{FJZ}.

Recall that a Lie superalgebra $\mathfrak{g}=\mathfrak{g}_{\bar{0}}\oplus \mathfrak{g}_{\bar{1}}$   is
 called restricted if $\mathfrak{g}_{\bar{0}}$ as a Lie algebra is  restricted   and
 $\mathfrak{g}_{\bar{1}}$ as a $\mathfrak{g}_{\bar{0}}$-module is  restricted.
 For a restricted Lie superalgebra $\mathfrak{g}$,  the $p$-mapping of $\mathfrak{g}_{\bar{0}}$ is also called
the $p$-mapping of the Lie superalgebra $\mathfrak{g}$.
A standard fact is that
$\frak{m}$ is  a restricted Lie superalgebra.
Note that the unique  $p$-mapping of $\frak{m}$ is the usual (associative) $p$-power.
Write
$$\|\underline{r}\|=\sum^{n}_{k=1}r_{k}+\sum^{n}_{k=1}r_{k'}+2r_{2n+1}.$$
Then $\frak{m}$  has a  $\mathbb{Z}$-grading
$\frak{m}=\oplus_{i\geq -2}\frak{m}_{[i]},$ which is said to be  standard,
given by
$$\frak{m}_{[i]}:=\mathrm{span}_{\mathbb{F}}\{\mathrm{M}_{x^{(\underline{r})}}
\mid \underline{r}\in \mathbf{I}(n,  n+1), \|\underline{r}\|=i+2\}, i\geq -2.$$
Hereafter denote by $\|\mathrm{M}_{x^{(\underline{r})}}\|$ the $\mathbb{Z}$-degree of $\mathrm{M}_{x^{(\underline{r})}}$,
which is $\|\underline{r}\|-2$.
There is a natural  filtration with respect to the standard $\mathbb{Z}$-grading:
$$\mbox{$\frak{m}=\frak{m}^{-2}\supseteq\frak{m}^{-1}\supseteq\cdots,$
where $\frak{m}^{i}:=\sum_{j\geq i}\frak{m}_{[j]}.$}$$
Let $\mathfrak{h}$  be the subalgebra of $\frak{m}$ with the basis
$$\{\mathrm{M}_{x^{(\varepsilon_{i}+\varepsilon_{i'})}}, \mathrm{M}_{x_{2n+1}}\mid i=1, \ldots, n\}.$$
Then $\mathfrak{h}$ is a  Cartan subalgebra of $\frak{m}$
and $\frak{m}=\mathfrak{h}\oplus\bigoplus_{0\neq \alpha\in \mathfrak{h}^{*}}\frak{m}_{\alpha},$
where
$$\frak{m}_{\alpha}:=\mathrm{span}_{\mathbb{F}}\{x\in \frak{m}\mid [h, x]=\alpha(h)x, \forall \; h\in \mathfrak{h}\}.$$
Note that $\frak{m}_{[0]}$ has a triangular decomposition
$\frak{m}_{[0]}=\mathfrak{n}_{[0]}^{-}\oplus \mathfrak{h}\oplus\mathfrak{n}_{[0]}^{+},$
where
\begin{eqnarray}\label{1245}
\mathfrak{n}_{[0]}^{-}:=\mathrm{span}_{\mathbb{F}}\{
\mathrm{M}_{x^{(\varepsilon_{i}+\varepsilon_{j'})}}\mid 1\leq j<i\leq n\}+\mathrm{span}_{\mathbb{F}}\{\mathrm{M}_{x^{(\varepsilon_{k'}+\varepsilon_{l'})}}\mid
k, l=1,\ldots, n\},
\end{eqnarray}
\begin{eqnarray}\label{1246}
\mathfrak{n}_{[0]}^{+}:=\mathrm{span}_{\mathbb{F}}\{\mathrm{M}_{x^{(\varepsilon_{i}+\varepsilon_{j'})}}\mid 1\leq i <j\leq n \}+\mathrm{span}_{\mathbb{F}}\{\mathrm{M}_{x^{(\varepsilon_{k}+\varepsilon_{l})}}\mid k, l=1, \ldots, n\}.
\end{eqnarray}
Put $\mathfrak{b}_{[0]}=\mathfrak{h}\oplus \mathfrak{n}_{[0]}^{+},$
which is called the standard Borel subalgebra of $\frak{m}_{[0]}$.

\subsection{Special contact Lie superalgebras of   odd type}

For any $\kappa\in \mathbb{F}$,
let
$$
\mathrm{div}_{\kappa}(f)=(-1)^{|f|}2\left(\Delta
\left(f\right)+\left(E-n\kappa\mathrm{id}\right)
\partial_{2n+1}\left(f\right)\right), ~ f\in \mathcal{O}(n,  n+1),
$$
where
$\Delta=\sum_{i=1}^{n}\partial_{i}\partial_{i'}$
is the odd Laplacian operator.
Then
$$
\{\mathrm{M}_{f}\in \frak{m}\mid \mathrm{div}_{\kappa}(f)=0\}
$$
 is a  subsuperalgebra of $\frak{m}$,
which is called   the  special contact   Lie superalgebra of   odd type,
denoted by $\frak{sm}(\kappa)$.
Its second derived algebra $\frak{sm}(\kappa)^{(2)}$ is simple when $n\geq 3$.
Note that $\frak{sm}(\kappa)$ is also called
  the special odd contact
superalgebra, and denoted by $SKO(n; \underline{1})$
 in  \cite[p. 72]{BL}.
Since $\mathrm{div}_{\kappa}$ is homogeneous
with respect to the standard $\mathbb{Z}$-grading of $\frak{m}$,
$\frak{sm}(\kappa)$  is a
 $\mathbb{Z}$-graded  subsuperalgebra  of $\frak{m}$ and
has also the corresponding filtration.
Let $\mathfrak{t}$ be the subalgebra of $\frak{sm}(\kappa)$ with the basis
$$\left\{\mathrm{M}_{x^{(\varepsilon_{i}+\varepsilon_{i'})}-x^{(\varepsilon_{i+1}+\varepsilon_{(i+1)'})}},  \mathrm{M}_{x_{2n+1}+n\kappa x^{(\varepsilon_{1}+\varepsilon_{1'})}}\mid i=1, \ldots, n-1\right\}.$$
Then $\mathfrak{t}$ is a Cartan subalgebra of  $\frak{sm}(\kappa)$ and
 $\frak{sm}(\kappa)=\mathfrak{t}\oplus\bigoplus_{0\neq \alpha\in \mathfrak{t}^{*}}\frak{sm}_{\alpha}(\kappa),$
where
$$\frak{sm}_{\alpha}(\kappa):=\mathrm{span}_{\mathbb{F}}\left\{x\in \frak{sm}(\kappa)\mid [h, x]=\alpha(h)x, \forall \; h\in \mathfrak{t}\right\}.$$
Note that $\frak{sm}(\kappa)_{[0]}$ has a triangular decomposition
 $\frak{sm}(\kappa)_{[0]}=\mathfrak{\bar{n}}_{[0]}^{-}\oplus \mathfrak{t}\oplus\mathfrak{\bar{n}}_{[0]}^{+},$
where $\mathfrak{\bar{n}}_{[0]}^{-}=\mathfrak{n}_{[0]}^{-}$ and $\mathfrak{\bar{n}}_{[0]}^{+}=\mathfrak{n}_{[0]}^{+}$
(defined in (\ref{1245}) and (\ref{1246}).)
Put $\mathfrak{\bar{b}}_{[0]}=\mathfrak{t}\oplus \mathfrak{n}_{[0]}^{+},$
which is called the standard  Borel subalgebra of $\frak{sm}(\kappa)_{[0]}$.

\subsection{Remark}
For convenience to  readers, below we  characterize $\mathbb{Z}$-grading of  $\frak{m}$ and $\frak{sm}$ in the cases of $n=1$ and $p=5$:
$\frak{m}=\oplus^{5}_{i=-2}\frak{m}_{[i]}$, and $\frak{sm}=\oplus^{3}_{i=-2}\frak{sm}_{[i]},$
where
$$
\frak{m}_{[-2]}=\frak{sm}_{[-2]} = \mathrm{span}_{\mathbb{F}}\{\mathrm{M}_{1}\},
\frak{m}_{[-1]}=\frak{sm}_{[-1]}= \mathrm{span}_{\mathbb{F}}\{\mathrm{M}_{x_{1}}, \mathrm{M}_{x_{2}}\},
$$
$$
\frak{m}_{[0]}= \mathrm{span}_{\mathbb{F}}\{\mathrm{M}_{x^{(2\varepsilon_{1})}}, \mathrm{M}_{x^{(\varepsilon_{1}+\varepsilon_{2})}}, \mathrm{M}_{x_{3}}\},
\frak{sm}_{[0]} = \mathrm{span}_{\mathbb{F}}\{\mathrm{M}_{x^{(2\varepsilon_{1})}}, \mathrm{M}_{x_{3}+n\kappa x^{(\varepsilon_{1}+\varepsilon_{2})}}\},
$$
$$
\frak{m}_{[1]} = \mathrm{span}_{\mathbb{F}}\{\mathrm{M}_{x^{(3\varepsilon_{1})}}, \mathrm{M}_{x^{(2\varepsilon_{1}+\varepsilon_{2})}}, \mathrm{M}_{x^{(\varepsilon_{1}+\varepsilon_{3})}}, \mathrm{M}_{x^{(\varepsilon_{2}+\varepsilon_{3})}}\},
$$
$$
\frak{sm}_{[1]} = \mathrm{span}_{\mathbb{F}}\{\mathrm{M}_{x^{(3\varepsilon_{1})}}, \mathrm{M}_{x^{(\varepsilon_{1}+\varepsilon_{3})}+(n\kappa-1)x^{(2\varepsilon_{1}+\varepsilon_{2})}}\},
$$
$$
\frak{m}_{[2]} = \mathrm{span}_{\mathbb{F}}\{\mathrm{M}_{x^{(2\varepsilon_{1}+\varepsilon_{3})}}, \mathrm{M}_{x^{(\varepsilon_{1}+\varepsilon_{2}+\varepsilon_{3})}}, \mathrm{M}_{x^{(3\varepsilon_{1}+\varepsilon_{2})}}, \mathrm{M}_{x^{(4\varepsilon_{1})}}\},
$$
$$
\frak{sm}_{[2]} = \mathrm{span}_{\mathbb{F}}\{\mathrm{M}_{x^{(2\varepsilon_{1}+\varepsilon_{3})}+(n\kappa-2)x^{(3\varepsilon_{1}+\varepsilon_{2})}}, \mathrm{M}_{x^{(4\varepsilon_{1})}}\},
$$
$$
\frak{sm}_{[3]} = \mathrm{span}_{\mathbb{F}}\{\mathrm{M}_{x^{(3\varepsilon_{1}+\varepsilon_{3})}+(n\kappa-3)x^{(4\varepsilon_{1}+\varepsilon_{2})}}\},
$$
$$
\frak{m}_{[3]}= \mathrm{span}_{\mathbb{F}}\{\mathrm{M}_{x^{(3\varepsilon_{1}+\varepsilon_{3})}}, \mathrm{M}_{x^{(4\varepsilon_{1}+\varepsilon_{2})}}, \mathrm{M}_{x^{(2\varepsilon_{1}+\varepsilon_{2}+\varepsilon_{3})}}\},
$$
$$
\frak{m}_{[4]} = \mathrm{span}_{\mathbb{F}}\{\mathrm{M}_{x^{(3\varepsilon_{1}+\varepsilon_{2}+\varepsilon_{3})}}, \mathrm{M}_{x^{(4\varepsilon_{1}+\varepsilon_{3})}}\},
\frak{m}_{[5]}= \mathrm{span}_{\mathbb{F}}\{\mathrm{M}_{x^{(4\varepsilon_{1}+\varepsilon_{2}+\varepsilon_{3})}}\}.
$$
In fact the ones of right sides are the standard bases.

\subsection{$u_{\chi}(\frak{g})$-modules}
Let $\frak{g}$ be a finite-dimensional restricted Lie superalgebra
with a $p$-mapping $[p]$ and $\chi\in\frak{g}_{\overline{0}}^{*}$.
As in the case of restricted Lie algebras,
we  define the so-called $\chi$-reduced enveloping
superalgebra $u_{\chi}(\frak{g})$ to be the quotient of
the universal enveloping superalgebra  $U(\frak{g})$ by the two-sided ideal
generated by
$$\{x^{p}-x^{[p]}-\chi(x)^{p}\mid x\in \frak{g}_{\overline{0}}\}.$$
A  $\frak{g}$-module $V$  is said to be of a $p$-character $\chi$, if
$x^{p} v-x^{[p]} v=\chi(x)^{p} v$
for any $x\in \frak{g}_{\bar{0}}$ and $ v\in V.$
When $\chi= 0$, a $\frak{g}$-module of the $p$-character $\chi$ is also
said to be a restricted $\frak{g}$-module,
otherwise a non-restricted $\frak{g}$-module.
Parallel to the case of restricted Lie algebras,
 any simple $\frak{g}$-module is finite-dimensional and has a unique $p$-character   $\chi\in\frak{g}_{\overline{0}}^{*}$
 (see \cite{W-ZH, Z1}).
 All the $\frak{g}$-modules with the $p$-character $\chi$   constitute
a  full subcategory of the $\frak{g}$-module category,
which coincides with the $u_{\chi}(\frak{g})$-module category.
For any  $\chi\in\frak{g}_{\overline{0}}^{*},$
we can extend  $\chi$ to   an element in $\frak{g}^{*}$
by letting $\chi(\frak{g}_{\overline{1}})=0.$
Hereafter we always regard $\chi\in\frak{g}_{\overline{0}}^{*}$
as an element of $\frak{g}^{*}$ in this way.
Let $\frak{g}=\frak{g}^{-2}\supseteq\frak{g}^{-1}\supseteq\cdots$
be a filtration of $\frak{g}$  and $\chi\in\frak{g}_{\overline{0}}^{*}.$
Define the height $\mathrm{ht}(\chi)$ of $\chi$ by
$\mathrm{ht}(\chi)=\min\{i\mid \chi(\frak{g}^{i})=0\}.$

From now on, write $\frak{g}=\frak{m}$ or $\frak{sm}(\kappa)$.
Let $\chi\in\frak{g}_{\overline{0}}^{*}$ and $M$ be a simple $u_{\chi}(\frak{g}^{0})$-module.
Then the $\chi$-reduced Kac module (see \cite[p. 243]{Y2}) associated with $M$
is defined to be the induced module
$K_{\chi}(M)=u_{\chi}(\frak{g})\otimes_{u_{\chi}(\frak{g}^{0})}M.$
Note that
$$K_{\chi}(M)\stackrel{\mathrm{superspace}}{\cong} u_{\chi}\left(\frak{g}_{[-1]}\oplus\frak{g}_{[-2]}\right)\otimes_{\mathbb{F}}M.$$
Fix an ordered  basis of $\frak{g}_{[-1]}\oplus\frak{g}_{[-2]}:$
$$\mathrm{M}_{x_{1'}}, \mathrm{M}_{x_{2'}}, \ldots,
\mathrm{M}_{x_{n'}}, \mathrm{M}_{x_{n}},
\mathrm{M}_{x_{n-1}}, \ldots, \mathrm{M}_{x_{1}}, \mathrm{M}_{1}.$$
Let $s=(s_{1'}, s_{2'}, \ldots, s_{n'}, s_{n}, s_{n-1}, \ldots, s_{1}, s_{2n+1})$
be an $(2n+1)$-tuple of non-negative integers with
 \mbox{$s_{i}, s_{2n+1}\in\{0, 1\}$ and $s_{i'}=0, \ldots,  p-1$ for
$i=1, \ldots, n$.}
Then by the PBW theorem, any element
$ v$ of $ K_{\chi}(M)$ can be written
as $ v=\sum_{s}X^{s}\otimes v(s)$ uniquely,
where $v(s)\in M$
and
$$X^{s}=\mathrm{M}^{s_{1'}}_{x_{1'}}\mathrm{M}^{s_{2'}}_{x_{2'}}\cdots
\mathrm{M}^{s_{n'}}_{x_{n'}}\mathrm{M}^{s_{n}}_{x_{n}}
\mathrm{M}^{s_{n-1}}_{x_{n-1}}\cdots\mathrm{M}^{s_{1}}_{x_{1}}\mathrm{M}^{s_{2n+1}}_{1}.$$
Write $|s|=\sum^{n}_{i=1}(s_{i}+s_{i'}+s_{2n+1}).$
By abuse of language, for  a positive integer   $m$ below  denote by
$o(m)$  the summation $\sum_{|s|< m}X^{s}\otimes v(s)$ in $K_{\chi}(M).$

Below we give a basic lemma.

\begin{lemma}\label{maximal vector}
Let $ v=\sum_{|s|\leq m}X^{s}\otimes v(s)\in K_{\chi}(M)$
be an eigenvector of some homogeneous
element $x\in \frak{g}^{3}$
with the eigenvalue $a.$
Then for all $s$ with $|s|=m,$
the nonzero $v(s)$'s are eigenvectors of $x$
with the eigenvalue $a$ or $-a.$
\end{lemma}

\begin{proof}
Since $ v$ is a eigenvector of $x$ with the eigenvalue $a$,
we have
\begin{eqnarray*}
a v = x v
& = & \sum_{|s|\leq m}\pm X^{s}\otimes xv(s)\\
&&+  \sum_{|s|\leq m}\sum^{2n}_{i=1}\pm s_{i}X^{s-\varepsilon_{i}}\otimes [x, \mathrm{M}_{x_{i}}]v(s)\\
&&+  \sum_{|s|\leq m}\pm s_{2n+1}X^{s-\varepsilon_{2n+1}}\otimes [x, \mathrm{M}_{1}]v(s)
+o(m-1).
\end{eqnarray*}
Consequently,
$$a\sum_{|s|= m}X^{s}\otimes v(s)=\sum_{|s|= m}\pm X^{s}\otimes x v(s),$$
from which it follows that $x v(s)=\pm a v(s)$
for any $s$ with $|s|=m.$
\end{proof}

\section{Simple modules with nonsingular characters}

For convenience,
for $a_{i},b_{j}\in\frak{g}$ we put
$$
(a_{1}, \ldots, a_{k})^{\mathrm{t}}(b_{1}, \ldots, b_{l}):=
\begin{pmatrix}
[a_{1}, b_{1}] & [a_{1}, b_{2}] & \cdots & [a_{1}, b_{l}]\\
 [a_{2}, b_{1}] & [a_{2}, b_{2}] & \cdots & [a_{2}, b_{l}]\\
 \vdots & \vdots& \ddots & \vdots\\
 [a_{k}, b_{1}] & [a_{k}, b_{2}] & \cdots & [a_{k}, b_{l}]
\end{pmatrix}
$$
and
for $\chi\in\frak{g}_{\overline{0}}^{*}$
$$
\chi((a_{1}, \ldots, a_{k})^{\mathrm{t}}(b_{1}, \ldots, b_{l})):=
\begin{pmatrix}
\chi([a_{1}, b_{1}]) & \chi([a_{1}, b_{2}]) & \cdots & \chi([a_{1}, b_{l}])\\
 \chi([a_{2}, b_{1}]) & \chi([a_{2}, b_{2}]) & \cdots & \chi([a_{2}, b_{l}])\\
 \vdots & \vdots& \ddots & \vdots\\
 \chi([a_{k}, b_{1}]) & \chi([a_{k}, b_{2}]) & \cdots & \chi([a_{k}, b_{l}])
\end{pmatrix}.
$$

\begin{proposition}\label{si=0}
Let  $\chi\in\frak{g}_{\overline{0}}^{*}$ with $\mathrm{ht}(\chi)=h\geq 2,$
and $M$ a simple $u_{\chi}(\frak{g}^{0})$-module.
Then the following statements hold.

(1)
Every simple $u_{\chi}(\frak{g}^{h-1})$-submodule of
$K_{\chi}(M)$ is 1-dimensional.

(2)
Suppose
\begin{itemize}
\item
$\chi([\frak{g}_{[h+1]}, \frak{g}_{[-2]}])\neq 0;$
\item
there exist
 $f_{1}, f_{2}, \ldots, f_{r}\in \frak{g}_{[h]}$
for some $ 1\leq r\leq 2n$ such that
$$\chi((f_{1}, \ldots, f_{r})^{\mathrm{t}}(\mathrm{M}_{x_{i_{1}}}, \mathrm{M}_{x_{i_{2}}}, \ldots, \mathrm{M}_{x_{i_{2n}}}))=\left(A_{r}\mid 0\right)$$
with $A_{r}$  an invertible $r\times r$ matrix
and  $(i_{1}, i_{2}, \ldots, i_{2n})$  a rearrangement of
 $(1, 2, \ldots, 2n).$
\end{itemize}
Then for any  1-dimensional   $u_{\chi}(\frak{g}^{h-1})$-submodule $\mathbb{F} v$
of $K_{\chi}(M)$  with
$$ v =\sum_{|s|\leq m}X^{s}\otimes v(s),$$
 we have
$$s_{i_{1}}=s_{i_{2}}=\cdots=s_{i_{r}}=s_{2n+1}= 0$$
for all $s$ with $v(s)\neq 0$ and $|s|=m$.

\end{proposition}

\begin{proof}
(1)Let $V$ be a simple $u_{\chi}(\frak{g}^{h-1})$-submodule
of $K_{\chi}(M)$.
Then
$\chi([\frak{g}^{h-1}, \frak{g}^{h-1}])=0$
by $\mathrm{ht}(\chi)=h\geq 2$ and
$[\frak{g}^{h-1}, \frak{g}^{h-1}]\subseteq\frak{g}^{h}.$
Since  the filtration associated with the standard
$\mathbb{Z}$-grading is restricted,  every element of
$[\frak{g}^{h-1}, \frak{g}^{h-1}]$ acts nilpotently on $V$.
By \cite[Lemma 3.3]{Y2},
$V$ is 1-dimensional.

(2)Let $V=\mathbb{F} v$ be any  1-dimensional   $u_{\chi}(\frak{g}^{h-1})$-submodule
of $K_{\chi}(M)$  with $ v =\sum_{|s|\leq m}X^{s}\otimes v(s).$
Since $\chi([\frak{g}_{[h+1]}, \frak{g}_{[-2]}])\neq 0,$
there exists an element $g$ of $\frak{g}_{[h+1]}$ such that
$[g, \mathrm{M}_{1}]\neq 0.$
Applying $g$  to $ v\in V$   and considering that
 $g\in\frak{g}_{[h+1]}$ acts nilpotently on $V$, we have
\begin{eqnarray*}
0 = g v
& = & \sum_{|s|\leq m}\pm X^{s}\otimes gv(s)+ \sum_{|s|\leq m}\sum^{2n}_{i=1}\pm s_{i}X^{s-\varepsilon_{i}}\otimes [g, \mathrm{M}_{x_{i}}]v(s)\\
&& + \sum_{|s|\leq m}\pm s_{2n+1}X^{s-\varepsilon_{2n+1}}\otimes [g, \mathrm{M}_{1}]v(s)
+o(m-1).
\end{eqnarray*}
Note that both $\frak{g}^{h}$ and $\frak{g}^{h+1}$
are $p$-unipotent restricted ideals of $\frak{g}^{0}$
and $\mathrm{ht}(\chi)=h.$
By \cite[Lemma 3.5]{Y2},
we have $gM=0 \mbox{\;and\;} [g, \mathrm{M}_{x_{i}}]M=0.$
Hence, we have
$$\sum_{|s|= m}\pm s_{2n+1}X^{s-\varepsilon_{2n+1}}\otimes [g, \mathrm{M}_{1}]v(s)=0.$$
Since $[g, \mathrm{M}_{1}]\in \frak{g}_{[h-1]}\subseteq\frak{g}^{h-1},$
$[g, \mathrm{M}_{1}]$ acts on $ v$ by a scalar.
And by $[g, \mathrm{M}_{1}]^{[p]^{k}}=0$ for some $k\in \mathbb{N}$,
we get that $[g, \mathrm{M}_{1}]v=\chi([g, \mathrm{M}_{1}]).$
Then by Lemma \ref{maximal vector},
we have that $[g, \mathrm{M}_{1}]v(s)=\chi([g, \mathrm{M}_{1}])v(s)$ or $-\chi([g, \mathrm{M}_{1}])v(s)$
for all $s$ with  $v(s)\neq 0$ and $|s|=m.$
It follows that
$\sum_{|s|= m}\pm s_{2n+1}X^{s-\varepsilon_{2n+1}}\otimes \chi([g, \mathrm{M}_{1}])v(s)=0.$
Consequently, we get that
$ s_{2n+1}\chi([g, \mathrm{M}_{1}])v(s)=0$ for $s$ with  $|s|=m.$
This implies that $s_{2n+1}=0$ for all $s$ with $v(s)\neq 0$ and   $|s|=m.$
Hence we obtain that $ v=\sum_{|s|\leq m}X^{s}\otimes v(s),$
where  $s_{2n+1}=0$
for all $s$ with  $|s|=m.$

Without loss of generality,
we may assume that
$$\chi((f_{1}, \ldots, f_{r})^{\mathrm{t}}(\mathrm{M}_{x_{1}}, \ldots, \mathrm{M}_{x_{2n}}))=\left(A_{r}\mid 0\right).$$
For $j=1, \ldots, r$,
applying $f_{j}$  to $ v\in V$ and considering that
 $f_{j}\in\frak{g}_{[h]}$ acts nilpotently on $V$, we have
 \begin{eqnarray*}
0 = f_{j} v
&=&  \sum_{|s|\leq m}\pm X^{s}\otimes f_{j}v(s)+ \sum_{|s|\leq m}\sum^{2n}_{i=1}\pm s_{i}X^{s-\varepsilon_{i}}\otimes [f_{j}, \mathrm{M}_{x_{i}}]v(s)\\
&&+ \sum_{|s|\leq m}\pm s_{2n+1}X^{s-\varepsilon_{2n+1}}\otimes [f_{j}, \mathrm{M}_{1}]v(s)+o(m-1).
\end{eqnarray*}
Note that $\frak{g}^{h}$
is a  $p$-unipotent restricted ideal of $\frak{g}^{0}$
and $\mathrm{ht}(\chi)=h.$
Then by \cite[Lemma 3.5]{Y2},
we have $\frak{g}^{h}M=0.$
Then, keeping in mind that $s_{2n+1}=0$ for all $s$ with $v(s)\neq 0$ and   $|s|=m,$
we have
$$f_{j} v=\sum_{|s|= m}\sum^{2n}_{i=1}\pm s_{i}X^{s-\varepsilon_{i}}\otimes [f_{j}, \mathrm{M}_{x_{i}}]v(s)+o(m-1).$$
Hence
 $\sum_{|s|= m}\sum^{2n}_{i=1}\pm s_{i}X^{s-\varepsilon_{i}}\otimes [f_{j}, \mathrm{M}_{x_{i}}]v(s)=0.$
A similar  argument shows that
$[f_{j}, \mathrm{M}_{x_{i}}]v(s)=\chi([f_{j}, \mathrm{M}_{x_{i}}])v(s)$ for all $s$ with $|s|=m.$
It follows that
\begin{equation}\label{equality}
\sum_{|s|=m}\sum^{2n}_{i=1}\pm s_{i}X^{s-\varepsilon_{i}}\otimes\chi([f_{j}, \mathrm{M}_{x_{i}}])v(s)=0
\end{equation}
for every $j=1, \ldots, r.$
We claim that $s_{i}=0$ for all $i=1, \ldots, r$ and $s$ with $v(s)\neq 0$ and  $|s|=m.$
Indeed, by (\ref{equality}), for every $j=1, \ldots, r$
we have
$$\sum_{|t|=m-1}\sum^{r}_{i=1}\pm (t_{i}+1)X^{\mathrm{t}}\otimes\chi([f_{j}, \mathrm{M}_{x_{i}}])u(t+\varepsilon_{i})=0.$$
Note that  the matrix
$\chi((f_{1}, \ldots, f_{r})^{\mathrm{t}}(\mathrm{M}_{x_{1}}, \ldots, \mathrm{M}_{x_{r}}))$
is invertible.
One sees that
$(t_{i}+1)u(t+\varepsilon_{i})=0$
for every $i=1, \ldots, r.$
Therefore, we obtain that $s_{i}v(s)=0$ for all $s$
with $v(s)\neq 0$ and $|s|=m$,
where $i=1, \ldots, r.$
Hence
$s_{1}=s_{2}=\cdots=s_{r}=0$ for all $s$ with $v(s)\neq 0$ and $|s|=m.$
\end{proof}

Similar to \cite[p. 413]{ZH(Cartan-type)}, we introduce the following definition.

\begin{definition}\label{1439}
Let $\chi\in\frak{g}_{\overline{0}}^{*}$ with $\mathrm{ht}(\chi)=h\geq 2.$

(1)
The  $\mathrm{rank}(\chi)$ of $\chi$ is defined
to be $\mathrm{rank}(\chi(A))$ for any matrix
$$A:=(f_{1}, \ldots, f_{k})^{\mathrm{t}}(\mathrm{M}_{x_{1}}, \ldots, \mathrm{M}_{x_{2n}})
\oplus(g_{1}, \ldots, g_{l})^{\mathrm{t}}(\mathrm{M}_{1}),$$
where $\{f_{1}, \ldots, f_{k}\}$ is a basis of $\frak{g}_{[h]}$
and $\{g_{1}, \ldots, g_{l}\}$
is  a basis of  $\frak{g}_{[h+1]}$.

(2)
$\chi$ is called a nonsingular (resp. singular) $p$-character  if  $\mathrm{rank} (\chi)=2n+1$ (resp. $\mathrm{rank} (\chi)\neq 2n+1$).

\end{definition}

Note that
$\mathrm{rank} (\chi)$  is
independent of the choice of the base for $\frak{g}_{[h]}$
and $\frak{g}_{[h+1]}$.
Below we give two  examples  about  nonsingular and singular $p$-characters.

\begin{example}
Let $\chi\in\frak{m}_{\bar{0}}^{*}$ with $2\leq \mathrm{ht}(\chi)=h<p-2$.
For  $n+1\leq i\leq 2n-1$, write
$$
\Sigma_{i}:= \{\mathrm{M}_{x^{(\underline{s})}}\in\frak{m}_{[h-1]}\mid
\chi(\mathrm{M}_{x^{(\underline{s})}})\neq 0, s_{n+1}=s_{n+2}=\cdots=s_{i}=0\},
$$
and for $j=2n, 2n+1,$ write
$$
\Sigma_{j}:= \{\mathrm{M}_{x^{(\underline{s})}}\in\frak{m}_{[h-1]}\mid
\chi(\mathrm{M}_{x^{(\underline{s})}})\neq 0, s_{j}=0\}.
$$
If every $\Sigma_{k}$ for  $n+1\leq k\leq 2n+1$ is not empty, then
$\chi$ is nonsingular.
\end{example}

\begin{proof}
Define a partial order on the  standard basis  of
$\frak{m}_{[h-1]}$ as follows:
$$\mathrm{M}_{x^{(\underline{r})}}<\mathrm{M}_{x^{(\underline{s})}}\Leftrightarrow
r_{2n+1}<s_{2n+1} \mbox{\;or\; } r_{2n+1}=s_{2n+1}, r_{1}=s_{1}, \ldots, r_{i-1}=s_{i-1}, r_{i}>s_{i}$$
for some $i=1, \ldots, 2n $,
where
$\|\underline{r}\|=\|\underline{s}\|=h+1.$
For each $i=n+1, \ldots, 2n+1$, let $\mathrm{M}_{x^{(\underline{r^{i}})}}$ be a minimal element of the set $\Sigma_{i}$
 under this partial order.
Denote by $\chi(\mathrm{M}_{x^{(\underline{r^{i}})}})=c_{i}\neq 0.$
Note that $\|\mathrm{M}_{x^{(\underline{r^{i}})}}\| =h-1<p-3$ and
so $r^{i}_{j}<p-1$ for every $j=1, \ldots, n$ and $i=n+1, \ldots, 2n+1.$
For simplicity, let $\mathrm{M}_{x^{(\underline{r^{k}})}}:=\mathrm{M}_{x^{(\underline{r^{n+1}})}}$ for each $k=1, \ldots, n$.
Hence the following matrix is nonzero:
\begin{eqnarray*}
A'&=&\chi((\mathrm{M}_{x^{(\underline{r^{1}}+\varepsilon_{1})}}, \ldots,
\mathrm{M}_{x^{(\underline{r^{2n}}+\varepsilon_{2n})}})^{\mathrm{t}}
(\mathrm{M}_{x_{n+1}}, \ldots, \mathrm{M}_{x_{2n}},
\mathrm{M}_{x_{1}}, \ldots, \mathrm{M}_{x_{n}}))\\
&&\oplus\chi([\mathrm{M}_{x^{(\underline{r^{2n+1}}+\varepsilon_{2n+1})}}, \mathrm{M}_{1}])
\end{eqnarray*}
By the definition of $\Sigma_{i}$ and the minimality of $\mathrm{M}_{x^{(\underline{r^{i}})}}$ in $\Sigma_{i}$,
one sees that $A'$ is a lower triangular matrix with
the diagonal entries $\pm c_{n+1}, \ldots, \pm c_{2n+1}$.
Hence $\chi$ is nonsingular.
\end{proof}

\begin{example}
Fix $h= p-2$. We take  $\chi\in \frak{m}_{\bar{0}}^{*}$ such that
\begin{enumerate}
\item[(1)] $\chi|_{\frak{g}^{h}}=0$;
\item[(2)] $\chi(\mathrm{M}_{x^{(\underline{r})}})=\delta_{r_{2n}, p-1}$ with $\| x^{(\underline{r})}\|=h+1$;
\item[(3)] $\chi|_{\frak{g}_{[i]}}=0$ for every $2\leq i\leq h-2$.
\end{enumerate}
Then it is easy to check that
$\chi([\frak{g}_{[h]}, \mathrm{M}_{x_{n}}])=0,$
i.e., $\chi$ is singular.
\end{example}

\begin{corollary}\label{1609}

Let $\chi\in\frak{g}_{\overline{0}}^{*}$ and $M$   a simple $u_{\chi}(\frak{g}^{0})$-module.
If $\chi$ is nonsingular,
then $1\otimes M$ is the unique  simple $u_{\chi}(\frak{g}^{0})$-submodule of   $K_{\chi}(M)$.

\end{corollary}
\begin{proof}
Let $V$ be a simple $u_{\chi}(\frak{g}^{0})$-submodule
of $K_{\chi}(M)$, and $V'$ a simple $u_{\chi}(\frak{g}^{h-1})$-submodule
of $V.$
By Definition \ref{1439}, $\mathrm{ht}(\chi)\geq 2$
and there exist   $f_{1}, \ldots, f_{2n}\in\frak{g}_{[h]}$
and $g\in \frak{g}_{[h+1]},$ such that
$\chi(A)=\chi((f_{1}, \ldots, f_{2n})^{\mathrm{t}}(\mathrm{M}_{x_{1}}, \ldots, \mathrm{M}_{x_{2n}}))$
is invertible and $\chi([g, \mathrm{M}_{1}])\neq 0.$
Then it follows from  Proposition \ref{si=0} that  $V'=\mathbb{F} v,$
where $ v=\sum_{|s|\leq m}X^{s}\otimes v(s)$
 with $s_{i}=0, i=1, \ldots, 2n+1$ for all $s$ with $|s|=m.$
This implies  $m=0,$
that is,   $ v\in 1\otimes M$.
Consequently,   $1\otimes M=V$, since both
 $1\otimes M$ and $V$ are simple $\frak{g}^{0}$-modules.
Thus  $1\otimes M$ is the unique simple $u_{\chi}(\frak{g}^{0})$-submodule of $K_{\chi}(M)$.
\end{proof}

We are in the position to state the main result of this section, which characterizes all simple $\frak{g}$-modules of nonsingular  $p$-characters.
\begin{theorem}\label{theorem1}
Suppose $\chi\in\frak{g}_{\overline{0}}^{*}$ is nonsingular.
Then the following statements hold.

(1)
$K_{\chi}(M)$ is simple for any simple $u_{\chi}(\frak{g}^{0})$-module $M$.

(2)
For any two simple $u_{\chi}(\frak{g}^{0})$-modules $M$ and $N$,
$$K_{\chi}(M)\cong K_{\chi}(N)\Longleftrightarrow M\cong N.$$

(3) Any simple $u_{\chi}(\frak{g})$-module is isomorphic to
some $\chi$-reduced Kac module $K_{\chi}(M)$ with $M$ some simple $u_{\chi}(\frak{g}^{0})$-module.

\end{theorem}

\begin{proof}
(1)
Let $K$ be any nonzero $u_{\chi}(\frak{g})$-submodule of   $K_{\chi}(M)$.
Then by Corollary \ref{1609}, $1\otimes M$ is also a simple $u_{\chi}(\frak{g}^{0})$-submodule of $K$.
Hence    $u_{\chi}(\frak{g})(1\otimes M)\subseteq K.$
Note that $u_{\chi}(\frak{g})(1\otimes M)=K_{\chi}(M),$
which forces $K=K_{\chi}(M)$.
Hence $K_{\chi}(M)$ is simple.

(2)
Assume that $K_{\chi}(M)\cong K_{\chi}(N).$
 By Corollary \ref{1609}, $1\otimes M$ (resp. $1\otimes N$) is the unique simple $u_{\chi}(\frak{g}^{0})$-submodule
of $K_{\chi}(M)$ (resp. $K_{\chi}(N)$).
Consequently,   $M\cong N$.

(3)
Let $A$ be any simple  $u_{\chi}(\frak{g})$-module.
Take a simple $u_{\chi}(\frak{g}^{0})$-submodule
of $A,$ denoted by $M.$
Hence there exists a canonical surjective morphism
$$\varphi: K_{\chi}(M)=u_{\chi}(\frak{g})\otimes_{u_{\chi}(\frak{g}^{0})}M\longrightarrow A.$$
Since both $K_{\chi}(M)$ and $A$ are simple $u_{\chi}(\frak{g})$-modules,
$\varphi$ is an isomorphism.
\end{proof}

\section{Simple modules of singular $p$-characters}
Let $\mathrm{Aut}(\frak{g})$ be the automorphism group of $\frak{g}.$
Then $\mathrm{Aut}(\frak{g})$ acts naturally on $\frak{g}^{*}$
by the coadjoint action, i.e., if $\Phi\in \mathrm{Aut}(\frak{g}),$
then
$(\Phi\cdot\chi)(x)=\chi(\Phi^{-1}x)$
for any $\chi\in\frak{g}_{\overline{0}}^{*}, x\in\frak{g}.$
Similar to the case of Lie algebras of Cartan type,
$\mathrm{Aut}(\frak{g})$ preserves the filtration of $\frak{g}$
and
it is easy to check that $\mathrm{ht}(\chi)=\mathrm{ht}(\Phi\cdot\chi)$
for any $\Phi\in \mathrm{Aut}(\frak{g})$ and $\chi\in\frak{g}_{\overline{0}}^{*}.$
For any $u_{\chi}(\frak{g})$-module $M$, let
$M^{\Phi}$ be   $\frak{g}$-module having
$M$ as its underlying vector space and
$\frak{g}$-action given by $x\cdot m=(\Phi^{-1}x) m,$
for any $x\in \frak{g}$ and $m\in M$,
where the action on the right is the given one.
Then $M^{\Phi}$ is simple if and only if $M$ is simple.
If $M$ has a $p$-character  $\chi$, then $M^{\Phi}$
has the $p$-character  $\Phi\cdot\chi.$
Moreover, similar to the case of Lie algebras of Cartan type,
$\mathrm{Aut}(\frak{g})=\mathrm{Aut}^{*}(\frak{g})\ltimes \mathrm{Aut}_{1}(\frak{g}),$
where
$$\mathrm{Aut}^{*}(\frak{g})=\{\Phi\in \mathrm{Aut}(\frak{g})\mid \Phi(\frak{g}_{[i]})=\frak{g}_{[i]}, \forall i\}$$
and
$$\mathrm{Aut}_{1}(\frak{g})=\{\Phi\in \mathrm{Aut}(\frak{g})\mid (\Phi-\mathrm{id}_{\frak{g}})(\frak{g}_{[i]})\subseteq\frak{g}^{i+1}, \forall i\}.$$

\begin{proposition}\label{rank=}
Let $\chi\in\frak{g}_{\overline{0}}^{*}$ and $\Phi\in \mathrm{Aut}(\frak{g}).$
Then $\rank (\Phi\cdot \chi)=\rank \chi.$
\end{proposition}
\begin{proof}
Let  $\mathrm{ht}(\chi)=h.$
Write
$$\mbox{$A_{1}=(f_{1}, \ldots, f_{k})^{\mathrm{t}}(\mathrm{M}_{x_{1}}, \ldots, \mathrm{M}_{x_{2n}})$
and
$A_{2}=(g_{1}, \ldots, g_{l})^{\mathrm{t}}(\mathrm{M}_{1}),$}$$
where  $\{f_{1}, \ldots, f_{k}\}$ is a basis of $\frak{g}_{[h]}$
and $\{g_{1}, \ldots, g_{l}\}$
is  a basis of  $\frak{g}_{[h+1]}$.
If $\Phi$ is in $\mathrm{Aut}^{*}(\frak{g}),$ then
there exist invertible matrices $R, S, T$ and $0\neq c\in \mathbb{F}$
such that
\begin{eqnarray*}
&&\Phi^{-1}(f_{1}, \ldots, f_{k})=(f_{1}, \ldots, f_{k})R,\\
&&\Phi^{-1}(\mathrm{M}_{x_{1}}, \ldots, \mathrm{M}_{x_{2n}})=(\mathrm{M}_{x_{1}}, \ldots, \mathrm{M}_{x_{2n}})S,\\
&&\Phi^{-1}(g_{1}, \ldots, g_{l})=(g_{1}, \ldots, g_{l})T,\\
&&\Phi^{-1}(\mathrm{M}_{1})=(\mathrm{M}_{1})c.
\end{eqnarray*}
Then
\begin{eqnarray*}(\Phi\cdot\chi)(A_{1})
&=&\chi(\Phi^{-1}A_{1})\\
&=&\chi(\Phi^{-1}(f_{1}),  \ldots, \Phi^{-1}(f_{k}))^{\mathrm{t}}
(\Phi^{-1}(\mathrm{M}_{x_{1}}), \ldots, \Phi^{-1}(\mathrm{M}_{x_{2n}}))\\
&=&\chi(R^{\mathrm{t}}(f_{1}, \ldots, f_{k})^{\mathrm{t}}(\mathrm{M}_{x_{1}}, \ldots, \mathrm{M}_{x_{2n}})S)\\
&=&R^{\mathrm{t}}\chi(A_{1})S,
\end{eqnarray*}
\begin{eqnarray*}(\Phi\cdot\chi)(A_{2})
&=&\chi(\Phi^{-1}A_{2})\\
&=&\chi(\Phi^{-1}(g_{1}),  \ldots, \Phi^{-1}(g_{l}))^{\mathrm{t}}
(\Phi^{-1}(\mathrm{M}_{1})))\\
&=&\chi(T^{\mathrm{t}}(g_{1}, \ldots, g_{l})^{\mathrm{t}}(\mathrm{M}_{1})c)\\
&=&T^{\mathrm{t}}\chi(A_{2})c.
\end{eqnarray*}
Letting $A=A_{1}\oplus A_{2}$, we have
$$(\Phi\cdot\chi)(A)
=R^{\mathrm{t}}\chi(A_{1})S\oplus T^{\mathrm{t}}\chi(A_{2})c
=(R^{\mathrm{t}}\oplus T^{\mathrm{t}})\chi(A)(S\oplus c).$$
It is clear that $\rank (\Phi\cdot \chi)(A)=\rank \chi(A),$
since  $R, S, T$ are invertible and $ c\neq 0.$

If $\Phi\in \mathrm{Aut}_{1}(\frak{g}),$
then we have
\begin{eqnarray*}
(\Phi\cdot\chi)([f_{i}, \mathrm{M}_{x_{j}}])
&=&\chi([\Phi^{-1}f_{i}, \Phi^{-1}\mathrm{M}_{x_{j}}])\\
&=&\chi([f_{i}+u_{i}, \mathrm{M}_{x_{j}}+v_{i}])\\
&=&\chi([f_{i}, \mathrm{M}_{x_{j}}]),
\end{eqnarray*}
\begin{eqnarray*}
(\Phi\cdot\chi)([g_{i}, \mathrm{M}_{1}])
&=&\chi([\Phi^{-1}g_{i}, \Phi^{-1}\mathrm{M}_{1}])\\
&=&\chi([g_{i}+w_{i}, \mathrm{M}_{1}+a])\\
&=&\chi([g_{i}, \mathrm{M}_{1}]),
\end{eqnarray*}
where $u_{i}\in\frak{g}^{h+1}, v_{i}\in\frak{g}^{0},
w_{i}\in\frak{g}^{h+2}\mbox{\;and\;} a\in\frak{g}^{-1}.$
Hence, $(\Phi\cdot\chi)(A)=\chi(\Phi^{-1}A)=\chi(A).$
Consequently, $\rank (\Phi\cdot \chi)(A)=\rank \chi(A)$.
\end{proof}
The statements at the start of this section and the above proposition tell us that we should find a convenient representative for
the orbit of $\chi$ under the coadjoint action of the automorphism group of $\frak{g}.$
For our purposes, $\chi$ may often be replaced by this
$p$-character, with the advantage being that certain arguments are thus
simplified.

Let $\chi$ be singular and of rank $r+1(< 2n+1)$
and $\mathrm{ht}(\chi)=h.$
Recall the characteristic matrix of $\frak{g}$
associated with $\chi$ is defined to be
$$\chi(A)=\chi((f_{1}, \ldots, f_{k})^{\mathrm{t}}(\mathrm{M}_{x_{1}}, \ldots, \mathrm{M}_{x_{2n}})
\oplus(g_{1}, \ldots, g_{l})^{\mathrm{t}}(\mathrm{M}_{1})),$$
 where $\{f_{1}, \ldots, f_{k}\}$ is a basis of $\frak{g}_{[h]}$
and $\{g_{1}, \ldots, g_{l}\}$
is  a basis of  $\frak{g}_{[h+1]}$.
Then
 $\rank(\chi(A))=r+1.$
Below we always assume that
$\chi((g_{1}, \ldots, g_{l})^{\mathrm{t}}(\mathrm{M}_{1}))$ is nonzero.
Hence $\rank(\chi((f_{1}, \ldots, f_{k})^{\mathrm{t}}(\mathrm{M}_{x_{1}}, \ldots, \mathrm{M}_{x_{2n}})))=r.$
That is, there exists a subset
$\{i_{1}, \ldots, i_{r}\}$ of $\{1, \ldots, 2n\}$
such that the matrix $\chi((f_{1}, \ldots, f_{k})^{\mathrm{t}}(\mathrm{M}_{x_{i_{1}}}, \ldots, \mathrm{M}_{x_{i_{r}}}))$
has an invertible $r\times r$ minor.
\begin{definition}\label{1523}
Let $\chi\in\frak{g}_{\overline{0}}^{*}$ with $\mathrm{ht}(\chi)=h\geq 5$ and $\mathrm{rank}(\chi)=r+1<2n+1.$
We say $\chi$ to be $\Delta$-invertible, if there exists an element $\Phi$ of $\mathrm{Aut}^{*}(\frak{g})$
such that
\begin{enumerate}
\item[(1)] there exists a partition of the set $\{1, \ldots, 2n\}:$
$$\mbox{$I\cup J=\{1, \ldots, 2n\}$ and $I\cap J=\emptyset$},$$
where $I:=\{i_{1}, \dots, i_{r}\}\neq\{1, \ldots, 2n\}$ and $J:=\{i_{r+1}, \dots, i_{2n}\};$
\item[(2)] there exists a basis $\{f_{1}, \ldots, f_{k}\}$
of  $\frak{g}_{[h]}$ such that the matrix
$$(\Phi\cdot \chi)((f_{1}, \ldots, f_{k})^{\mathrm{t}}(\mathrm{M}_{x_{i_{1}}}, \ldots, \mathrm{M}_{x_{i_{r}}}))$$
has an invertible $r\times r$ minor;
\item[(3)] $(\Phi\cdot \chi)([\frak{g}_{[h]}, \mathrm{M}_{x_{j}}])=0$ for every  $j\in J$;
\item[(4)] there is a $\frak{g}_{[0]}$-submodule $\Delta\subseteq \frak{g}_{[h-1]}$ with $(\Phi\cdot \chi)(\Delta)=0$;
\item[(5)] there are homogeneous elements $e_{r+1}, \ldots, e_{2n}\in \Delta$ such that
the matrix
$$(\Phi\cdot \chi)((e_{r+1}, \ldots, e_{2n})^{\mathrm{t}}(\mathrm{M}_{x_{i_{r+1}}}, \ldots, \mathrm{M}_{x_{i_{2n}}}))$$
is invertible.
\end{enumerate}
\end{definition}

\begin{corollary}\label{1420}
Let $\chi\in\frak{g}_{\overline{0}}^{*}$ and $M$ be  a simple $u_{\chi}(\frak{g}^{0})$-module.
If  $\chi$ is $\Delta$-invertible,
then $1\otimes M$ is the unique simple $u_{\chi}(\frak{g}^{0})$-submodule of   $K_{\chi}(M)$.
\end{corollary}

\begin{proof}
Suppose $\mathrm{rank}(\chi)=r+1$.
By Proposition \ref{rank=} and the arguments before Definition \ref{1523},
without loss of generality, we may assume  $I=\{1, 2, \ldots, r\}$ and $\Phi=\mathrm{id}$.
Let $V$ be a simple $u_{\chi}(\frak{g}^{h-2})$-submodule
of $K_{\chi}(M)$.
We have $[\frak{g}^{h-2}, \frak{g}^{h-2}]\subseteq\frak{g}^{2h-5}\subseteq\frak{g}^{h},$
where the latter inclusion is implied by $h\geq 5$.
Since $\mathrm{ht}(\chi)=h,$
it follows that  $\chi([\frak{g}^{h-2}, \frak{g}^{h-2}])=0.$
Consequently, every element in $[\frak{g}^{h-2}, \frak{g}^{h-2}]$
is nilpotent. Then by \cite[Lemma 3.3]{Y2},
$V$ is 1-dimensional, denoted by
 $\mathbb{F} v,$
where $ v=\sum_{|s|\leq m}X^{s}\otimes v(s).$
Moreover, $V$ is also a simple $u_{\chi}(\frak{g}^{h-1})$-module.
By Definition \ref{1523} and Proposition \ref{si=0},
  we have
 $s_{1}=s_{2}=\cdots=s_{r}=s_{2n+1}= 0,$
for all $s$ with $v(s)\neq 0$ and $|s|=m$.

By Definition \ref{1523}, there exists a
$\frak{g}_{[0]}$-submodule $\Delta\subseteq \frak{g}_{[h-1]}$ with $\chi(\Delta)=0$.
Since $\mathrm{ht}(\chi)=h\geq 5$ and  the filtration  is restricted,
we get that $\Delta+\frak{g}^{h}$
is a $p$-unipotent restricted ideal of $\frak{g}^{0}$ on  which
$\chi$ vanishes.
Then by \cite[Lemma 3.5]{Y2}, we  have
$(\Delta+\frak{g}^{h})M=0.$
Let $e_{r+1}, \ldots, e_{2n}\in \Delta$ such that
the matrix $\chi((e_{r+1}, \ldots, e_{2n})^{\mathrm{t}}(\mathrm{M}_{x_{r+1}}, \ldots, \mathrm{M}_{x_{2n}}))$
is invertible.
For $r+1\leq j\leq 2n$, applying $e_{j}$  to $ v$,
we obtain that
\begin{eqnarray*}\label{e_{j}}
e_{j} v
&=&  \sum_{|s|\leq m}\pm X^{s}\otimes e_{j}v(s)\\
&&+ \sum_{|s|\leq m}\sum^{2n}_{i=1}\pm s_{i}X^{s-\varepsilon_{i}}\otimes [e_{j}, \mathrm{M}_{x_{i}}]v(s)\\
&&+ \sum_{|s|\leq m}\pm s_{2n+1}X^{s-\varepsilon_{2n+1}}\otimes [e_{j}, \mathrm{M}_{1}]v(s)+o(m-1)\\
&&=\sum_{|s|= m}\sum^{2n}_{i=r+1}\pm s_{i}X^{s-\varepsilon_{i}}\otimes [e_{j}, \mathrm{M}_{x_{i}}]v(s)+o(m-1).
\end{eqnarray*}
For the last equality we should note that
  $\Delta$ acts trivially on $M$ and
 $s_{1}=s_{2}=\cdots=s_{r}=s_{2n+1}=0$
for all $s$ with $v(s)\neq 0$ and $|s|=m$.
Since $\mathbb{F} v$ is a $\frak{g}^{h-1}$-module and
the homogeneous element  $e_{j}$ is in $\Delta\subseteq\frak{g}_{[h-1]},$
we may assume that $ v$ is a eigenvector of $e_{j}$ with the eigenvalue $a$.
By Lemma \ref{maximal vector}, for all $s$ with $|s|=m,$
the nonzero $v(s)$'s are eigenvectors of $e_{j}$
with the eigenvalue $a$ or $-a.$
Since $\Delta$ acts trivially on $M$, we have $a=0.$
Hence
$$\sum_{|s|= m}\sum^{2n}_{i=r+1}\pm s_{i}X^{s-\varepsilon_{i}}\otimes [e_{j}, \mathrm{M}_{x_{i}}]v(s)+o(m-1)=0.$$
Since $[e_{j}, \mathrm{M}_{x_{i}}]\in \frak{g}_{[h-2]}\subseteq\frak{g}^{h-2},$
$ v$ is an eigenvector of $[e_{j}, \mathrm{M}_{x_{i}}].$
By $[e_{j}, \mathrm{M}_{x_{i}}]^{[p]^{k}}=0$ for some $k\in \mathbb{N}$,
we get that $[e_{j}, \mathrm{M}_{x_{i}}]v(s)=\chi([e_{j}, \mathrm{M}_{x_{i}}])v(s)$
for all $s$ with $|s|=m.$
It follows that
\begin{equation}\label{equality1}
\sum_{|s|= m}\sum^{2n}_{i=r+1}\pm s_{i}X^{s-\varepsilon_{i}}\otimes \chi([e_{j}, \mathrm{M}_{x_{i}}])v(s)=0.
\end{equation}
Consequently,  $ s_{i}\chi([e_{j}, \mathrm{M}_{x_{i}}])v(s)=0$
for $s$ with  $|s|=m,$
and  $i, j= r+1, \ldots, 2n.$
We claim that $s_{i}=0, i= r+1, \ldots, 2n$ for all $s$ with $|s|=m.$
Indeed, by (\ref{equality1}), for every $j= r+1, \ldots, 2n,$
we have
$$\sum_{|t|=m-1}\sum^{2n}_{i=r+1}\pm (t_{i}+1)X^{\mathrm{t}}\otimes\chi([e_{j}, \mathrm{M}_{x_{i}}])u(t+\varepsilon_{i})=0.$$
Note that  the matrix
$\chi((e_{r+1}, \ldots, e_{2n})^{\mathrm{t}}(\mathrm{M}_{x_{r+1}}, \ldots, \mathrm{M}_{x_{2n}}))$
is invertible.
For every $i= r+1, \ldots, 2n,$
we have
$(t_{i}+1)u(t+\varepsilon_{i})=0$.
Therefore, $s_{i}v(s)=0$ for all $s$
with $v(s)\neq 0$ and $|s|=m$.
Furthermore,
$s_{i}=0, i= r+1, \ldots, 2n$ for all $s$ with $|s|=m.$
Consequently, $m=0,$
that is, $ v\in 1\otimes M.$
Thus  $1\otimes M$ is the unique $K_{\chi}(M)$
  $u_{\chi}(\frak{g}^{0})$-submodule.
\end{proof}

As for Theorem \ref{theorem1}, using Corollary \ref{1420} one may  prove the following theorem, which characterizes all simple $\frak{g}$-modules of  $\Delta$-invertible $p$-characters.
\begin{theorem}\label{1421}
Suppose   $\chi\in\frak{g}_{\overline{0}}^{*}$ is $\Delta$-invertible.
Then the following statements hold.

(1)
$K_{\chi}(M)$ is simple for any simple $u_{\chi}(\frak{g}^{0})$-module $M$.

(2)
For any two simple $u_{\chi}(\frak{g}^{0})$-modules  $M$ and $N$,
$$K_{\chi}(M)\cong K_{\chi}(N)\Longleftrightarrow M\cong N.$$

(3) Any simple $u_{\chi}(\frak{g})$-module is isomorphic to
 $K_{\chi}(M)$ for some simple $u_{\chi}(\frak{g}^{0})$-module $M$.
\end{theorem}

\section{Simple modules of regular semisimple $p$-characters of height one}

Let $\chi\in\frak{g}_{\overline{0}}^{*}$ with $\mathrm{ht}(\chi)\leq 1.$
Note that  $\frak{g}^{1}$ is a  $p$-unipotent restricted ideal of $\frak{g}^{0}.$
Then by \cite[Lemma 3.5]{Y2},
$\frak{g}^{1}$ acts trivially on any simple $\frak{g}^{0}$-module.
Consequently, any simple $u_{\chi}(\frak{g}^{0})$-module is a simple
$u_{\chi}(\frak{g}_{[0]})$-module. And
any simple $u_{\chi}(\frak{g}_{[0]})$-module can be
extended to be a simple $u_{\chi}(\frak{g}^{0})$-module
with the trivial action by $\frak{g}^{1}$.
Then we can obtain   the $\chi$-reduced Kac module $K_{\chi}(M)$
for a simple $u_{\chi}(\frak{g}_{[0]})$-module $M$.

For convenience, write $\mathrm{h}_{j}=\mathrm{M}_{x^{(\varepsilon_{j}+\varepsilon_{j'})}-x^{(\varepsilon_{j+1}+\varepsilon_{(j+1)'})}}$
for $j=1, \ldots, n-1.$
\begin{definition}
Let $\chi\in\frak{g}_{\overline{0}}^{*}$ with $\mathrm{ht}(\chi)= 1.$
$\chi$ is said to be regular semisimple,
if there exists an element $\Phi$ in $\mathrm{Aut}(\frak{g})$
such that

(1)
$(\Phi\cdot\chi)(\frak{n}^{\pm}_{[0]})=0, i=1, \ldots, n;$

(2)
$(\Phi\cdot\chi)(\mathrm{h}_{j})\neq 0, j=1, \ldots, n-1;$

(3)
$(\Phi\cdot\chi)(\sum^{n-1}_{j=1}\mathrm{h}_{j})\neq 0.$
\end{definition}

We are in the position to state the main result of this section, which gives all simple $\frak{g}$-modules of regular semisimple
$p$-characters.
\begin{theorem}
Let $\chi$ be a regular semisimple $p$-character.
Then the set
$$\{K_{\chi}(M)\mid \mbox{$M$ is a simple $u_{\chi}(\frak{g}_{[0]})$-module}\}$$
exhausts all simple $u_{\chi}(\frak{g})$-modules.
\end{theorem}

\begin{proof}
(1) Firstly we shall prove that $K_{\chi}(M)$ is simple
 for any simple $U_{\chi}(\frak{g}_{[0]})$-module $M$.
Without   loss of generality,
we may assume that
$\chi$ satisfies
$$\mbox{$\chi\left(\sum^{n-1}_{k=1}\mathrm{h}_{k}\right)\neq 0$ and $\chi(\mathrm{h}_{j})\neq 0$
for $j=1, \ldots, n-1.$}$$
Let $K$ be a nonzero $u_{\chi}(\frak{g})$-submodule
of $K_{\chi}(M)$.
Take a nonzero element $v\in K$ and suppose $v=\sum_{|s|\leq m}X^{s}\otimes v(s).$
Write
 $$y=\mathrm{M}^{p-1}_{x_{1'}}\mathrm{M}^{p-1}_{x_{2'}}\cdots
\mathrm{M}^{p-1}_{x_{n'}}\mathrm{M}_{x_{n}}
\mathrm{M}_{x_{n-1}}\cdots\mathrm{M}_{x_{1}}\mathrm{M}_{1}.$$
Then applying   $\mathrm{M}_{1}$ and $\mathrm{M}_{x_{i}}$
 to $ v$ for some $i=1, \ldots, 2n$,
we get   $y\otimes v(s)\in K$ for some nonzero  $v(s)\in M$.
For any $i\neq j=1, \ldots, 2n,$ we get
$[\mathrm{M}_{x^{(\varepsilon_{i}+\varepsilon_{j})}}, y]=0.$
Consequently, we have
$y\otimes  \mathrm{M}_{x^{(\varepsilon_{i}+\varepsilon_{j})}}  v(s)\in K.$
Note that $M$ is a simple  $u_{\chi}(\frak{g}_{[0]})$-module.
Hence, it is obvious that $y\otimes M$ is contained in $K.$
Assume that $ v_{\lambda}\in M$ is a highest weight
vector of the weight $\lambda$ with respect to the
standard Borel subalgebra of $\frak{g}_{[0]}.$
So  we get  $y\otimes  v_{\lambda}\in  K.$
Write
$$y^{\hat{1}}_{1'}=\mathrm{M}^{p-2}_{x_{1'}}\mathrm{M}^{p-1}_{x_{2'}}\cdots
\mathrm{M}^{p-1}_{x_{n'}}\mathrm{M}_{x_{n}}
\mathrm{M}_{x_{n-1}}\cdots\mathrm{M}_{x_{1}}\mathrm{M}_{1}.$$
Then applying
$\mathrm{M}_{x^{(\varepsilon_{1}+\varepsilon_{2}+\varepsilon_{2'})}-x^{(\varepsilon_{1}+\varepsilon_{3}+\varepsilon_{3'})}}$
to $y\otimes  v_{\lambda},$ we get
$(p-1)\lambda(\mathrm{h}_{2})y^{\hat{1}}_{1'}\otimes  v_{\lambda}\in  K.$
Since $M$ is a $u_{\chi}(\frak{g}_{[0]})$-module, for $i=1, \ldots, n-1,$ we have
$(\mathrm{h}^{p}_{i}-\mathrm{h}^{[p]}_{i}-\chi(\mathrm{h}_{i})^{p}) v_{\lambda}=0.$
Then by $\mathrm{h}^{[p]}_{i}=\mathrm{h}_{i}$, we get
$\lambda(\mathrm{h}_{i})^{p}-\lambda(\mathrm{h}_{i})=\chi(\mathrm{h}_{i})^{p}.$
Since
$\chi(\mathrm{h}_{i})\neq 0,$ we obtain
\begin{equation}\label{hi}
\lambda(\mathrm{h}_{i})\neq 0,\; i=1, \ldots, n-1.
\end{equation}
This implies that
$y^{\hat{1}}_{1'}\otimes  v_{\lambda}\in  K.$
Write
$$y^{\widehat{p-1}}_{1'}:=\mathrm{M}^{p-1}_{x_{2'}}\cdots
\mathrm{M}^{p-1}_{x_{n'}}\mathrm{M}_{x_{n}}
\mathrm{M}_{x_{n-1}}\cdots\mathrm{M}_{x_{1}}\mathrm{M}_{1}.$$
Then, applying  $\mathrm{M}^{p-2}_{x^{(\varepsilon_{1}+\varepsilon_{2}+\varepsilon_{2'})}-x^{(\varepsilon_{1}+\varepsilon_{3}+\varepsilon_{3'})}}$
 to $y^{\hat{1}}_{1'}\otimes  v_{\lambda}$, we obtain  from (\ref{hi}) that
    $y^{\widehat{p-1}}_{1'}\otimes  v_{\lambda}\in K$.
Write $$y^{\widehat{p-1}}_{(n-2)'}:=\mathrm{M}^{p-1}_{x_{(n-1)'}}\mathrm{M}^{p-1}_{x_{n'}}\mathrm{M}_{x_{n}}
\mathrm{M}_{x_{n-1}}\cdots\mathrm{M}_{x_{1}}\mathrm{M}_{1}.$$
Similarly, by applying
$(\mathrm{M}_{x^{(\varepsilon_{i}+\varepsilon_{i+1}+\varepsilon_{(i+1)'})}-x^{(\varepsilon_{i}+\varepsilon_{i+2}+\varepsilon_{(i+2)'})}})^{p-1}$
 in the order of $i=2, 3, \ldots, n-2$,  it follows from (\ref{hi}) that
$y^{\widehat{p-1}}_{(n-2)'}\otimes  v_{\lambda}\in K$.
Write
$$y^{\hat{1}}_{(n-1)'}:=\mathrm{M}^{p-2}_{x_{(n-1)'}}\mathrm{M}^{p-1}_{x_{n'}}\mathrm{M}_{x_{n}}
\mathrm{M}_{x_{n-1}}\cdots\mathrm{M}_{x_{1}}\mathrm{M}_{1}.$$
Applying  $\mathrm{M}_{x^{(\varepsilon_{n-1}+\varepsilon_{n}+\varepsilon_{n'})}-x^{(\varepsilon_{n-1}+\varepsilon_{1}+\varepsilon_{1'})}}$
to
$y^{\widehat{p-1}}_{(n-2)'}\otimes  v_{\lambda},$
we have
$$(p-1)\left(\lambda\left(\sum^{n-1}_{i=1}\mathrm{h}_{i}\right)+1\right)y^{\hat{1}}_{(n-1)'}\otimes  v_{\lambda}\in K.$$
Since $\left(\sum^{n-1}_{i=1}\mathrm{h}_{i}\right)^{p} v_{\lambda}-
\left(\sum^{n-1}_{i=1}\mathrm{h}_{i}\right)^{[p]} v_{\lambda}=
\chi\left(\sum^{n-1}_{i=1}\mathrm{h}_{i}\right)^{p} v_{\lambda}$
and $\left(\sum^{n-1}_{i=1}\mathrm{h}_{i}\right)^{[p]}=\sum^{n-1}_{i=1}\mathrm{h}_{i},$
we have $\lambda\left(\sum^{n-1}_{i=1}\mathrm{h}_{i}\right)^{p}-\lambda\left(\sum^{n-1}_{i=1}\mathrm{h}_{i}\right)=
\chi\left(\sum^{n-1}_{i=1}\mathrm{h}_{i}\right)^{p}.$
Considering that $\chi$ is regular semisimple,
we have $\chi(\sum^{n-1}_{i=1}\mathrm{h}_{i})\neq 0.$
Consequently, $\lambda(\sum^{n-1}_{i=1}\mathrm{h}_{i})\neq -1,$
so  $y^{\hat{1}}_{(n-1)'}\otimes  v_{\lambda}\in K.$
Applying   $\mathrm{M}^{p-2}_{x^{(\varepsilon_{n-1}+\varepsilon_{n}+\varepsilon_{n'})}-x^{(\varepsilon_{n-1}+\varepsilon_{1}+\varepsilon_{1'})}}$
to $y^{\hat{1}}_{(n-1)'}\otimes  v_{\lambda},$
we get   $y^{\widehat{p-1}}_{(n-1)'}\otimes  v_{\lambda}\in K$ where
 $y^{\widehat{p-1}}_{(n-1)'}:=\mathrm{M}^{p-1}_{x_{n'}}\mathrm{M}_{x_{n}}
\mathrm{M}_{x_{n-1}}\cdots\mathrm{M}_{x_{1}}\mathrm{M}_{1}.$
Applying
$\mathrm{M}^{p-1}_{x^{(\varepsilon_{n}+\varepsilon_{1}+\varepsilon_{1'})}-x^{(\varepsilon_{n}+\varepsilon_{2}+\varepsilon_{2'})}}$
to $y^{\widehat{p-1}}_{(n-1)'}\otimes  v_{\lambda}$, we obtain from  (\ref{hi}) that
$y^{\widehat{p-1}}_{n'}\otimes  v_{\lambda}\in K$
where $y^{\widehat{p-1}}_{n'}:=\mathrm{M}_{x_{n}}
\mathrm{M}_{x_{n-1}}\cdots\mathrm{M}_{x_{1}}\mathrm{M}_{1}.$

Applying
$\mathrm{M}_{x^{(\varepsilon_{i'}+\varepsilon_{1}+\varepsilon_{1'})}-x^{(\varepsilon_{i'}+\varepsilon_{2}+\varepsilon_{2'})}}$
to $y^{\widehat{p-1}}_{n'}\otimes  v_{\lambda}$
in the order of  $i=n, n-1, \ldots, 3 $ and (\ref{hi}),
we get
$\mathrm{M}_{x_{2}}\mathrm{M}_{x_{1}}\mathrm{M}_{1}
\otimes  v_{\lambda}\in K$.
Going on applying
$\mathrm{M}_{x^{(\varepsilon_{j'}+\varepsilon_{3}+\varepsilon_{3'})}-x^{(\varepsilon_{j'}+\varepsilon_{4}+\varepsilon_{4'})}}$
to $\mathrm{M}_{x_{2}}
\mathrm{M}_{x_{1}}\mathrm{M}_{1}
\otimes  v_{\lambda}$
in the order of  $j=2, 1$,
we get
$\mathrm{M}_{1}\otimes  v_{\lambda}\in K$
 by (\ref{hi}).
For $\frak{g}=\frak{m}$,
apply  $\mathrm{M}_{x^{(\varepsilon_{2n+1}+\varepsilon_{1}+\varepsilon_{1'})}-x^{(\varepsilon_{2n+1}+\varepsilon_{2}+\varepsilon_{2'})}}$
to $\mathrm{M}_{1}\otimes  v_{\lambda}.$
For $\frak{g}=\frak{sm}$, apply
$\mathrm{M}_{x^{(\varepsilon_{2n+1}+\varepsilon_{1}+\varepsilon_{1'})}-x^{(\varepsilon_{2n+1}+\varepsilon_{2}+\varepsilon_{2'})}
+(n\kappa-2)x^{(\varepsilon_{1}+\varepsilon_{1'}+\varepsilon_{2}+\varepsilon_{2'})}}$
to $\mathrm{M}_{1}\otimes  v_{\lambda}.$
Then we  get   $1\otimes  v_{\lambda}\in K$
 by (\ref{hi}).
Therefore $1\otimes M$ is contained in $K$, since   $u_{\chi}(\frak{g}_{[0]})$-module $M$ is simple.
Consequently, we get $K=K_{\chi}(M).$
Hence $K_{\chi}(M)$ is simple.

(2) Finally we shall prove that any simple $U_{\chi}(\frak{g})$-module is isomorphic to
 some $\chi$-reduced Kac module $K_{\chi}(M)$ with $M$ some simple $U_{\chi}(\frak{g}_{[0]})$-module.
Let $A$ be any simple  $u_{\chi}(\frak{g})$-module.
Take a simple $u_{\chi}(\frak{g}^{0})$-submodule
of $A,$ denoted by $M.$
Note that $M$ is also a simple $u_{\chi}(\frak{g}_{[0]})$-module.
Hence there exists a canonical surjective morphism
$\varphi: K_{\chi}(M)\longrightarrow A.$
Since both $K_{\chi}(M)$ and $A$ are simple,
$\varphi$ is an isomorphism.
\end{proof}




\end{document}